\documentclass[11pt]{amsart}
\usepackage{amsfonts,amssymb,amsmath,amsthm,here}
\usepackage{color}  

\theoremstyle{plain}
\newtheorem{theo}{Theorem}[section]
\newtheorem*{thm}{Theorem}

\newtheorem{cor}[theo]{Corollary}
\newtheorem{lem}[theo]{Lemma}
\newtheorem{prop}[theo]{Proposition}

\theoremstyle{definition}
\newtheorem{df}[theo]{Definition}
\newtheorem{ex}[theo]{Example}

\newtheorem{rem}[theo]{Remark}

\makeatletter

\@addtoreset{equation}{section}
\makeatother

\makeatletter
\@namedef{subjclassname@2020}{
  \textup{2020} Mathematics Subject Classification}
\makeatother

\def \calO{{\mathcal O}}

\def \Z{{\mathbb Z}}
\def \Q{{\mathbb Q}}

\def \e{\varepsilon}
\def \a{\alpha}

\author{Yu Hashimoto}

\author{Miho Aoki}

\address{
Department of Mathematics,
Interdisciplinary Faculty of Science and Engineering,
Shimane University,
1060, Nishikawatsu, Matsue, Shimane, 690-8504, Japan
 }
\email{yh70217021@gmail.com}

\address{Department of Mathematics,
Interdisciplinary Faculty of Science and Engineering,
Shimane University,
1060, Nishikawatsu, Matsue, Shimane, 690-8504, Japan
}
\email{aoki@riko.shimane-u.ac.jp }

\thanks{This work was supported by JSPS KAKENHI Grant Number JP21K03181}  

\title[Simplest cubic fields]{Normal integral bases and Gaussian periods \\ in the simplest cubic fields}  

\begin{document}

\begin{abstract} 
We give all normal integral bases  for the simplest cubic field $L_n$  generated by the roots of Shanks' cubic polynomial 
when these bases exist, that is, $L_n/\mathbb Q$ is tamely ramified. Furthermore, as an application of the result,
we give an explicit relation  between the roots of Shanks' cubic polynomial  and the Gaussian periods of $L_n$ 
in the case $L_n/\mathbb Q$ is tamely ramified, which is a generalization of the work of Lehmer, Ch\^{a}telet and Lazarus  
 in the case that the conductor of $L_n$ is equal to $n^2+3n+9$.
 
%
\end{abstract}
\maketitle
\section{Introduction}\label{sect:intro}

Let $n$ be an integer.
We consider Shanks' cubic polynomial $f_n(X)\in\Z[X]$ defined by
\begin{equation}\label{eq:fn}
f_n(X)=X^3-nX^2-(n+3)X-1 ,
\end{equation}
see \cite[p.\ 1138]{S}. 
The polynomial $f_n(X)$ is irreducible for all $n\in\Z$,
and $L_n:=\Q(\rho_n)$ is a cyclic cubic field where $\rho_n$ is a root of 
$f_n(X)$. 
The field is called the simplest cubic field. 
We have $L_n=L_{-n-3}$ since $f_n(X)=-X^3f_{-n-3}(1/X)$.
We write $\Delta_n:=n^2+3n+9=bc^3$ 
with $b,c\in\Z_{>0}$ where $b$ is cube-free.
$L_n/\Q$ is tamely ramified if and only if $3$ is unramified.
The discriminant of $f_n(X)$ is $d(f_n)=\Delta_n^2$, 
and the discriminant $D_{L_n}$ 
of $L_n$ has been known as follows
(see Cusick \cite[Lemma 1]{C} in the case that $\Delta_n$ is square-free,
Washington \cite[Proposition 1]{WL} in the case of $n\not\equiv3 \pmod{9}$, and Kashio-Sekigawa \cite{KS} in general).
\begin{equation}\label{eq:Df}
D_{L_n}=\mathfrak{f}_{L_n}^2
\end{equation}
where $\mathfrak{f}_{L_n}$ is the conductor of $L_n$ given by
\begin{equation}\label{eq:cond}
\mathfrak{f}_{L_n}=\gamma\prod_{\substack{
p\mid b
\\
p\neq3
}}p,\ \ 
\gamma=
\begin{cases} 
1, & \text{if}\ 3\nmid n\ \text{or}\ n\equiv12 \pmod{27},
\\
3^2, & \text{otherwise},
\end{cases}
\end{equation}
where the product runs over all prime numbers $p$ dividing $b$.
Especially $L_n/\mathbb{Q}$ is tamely ramified if and only if $3\nmid n$
or $n\equiv12 \pmod{27}$.
Note that $\Delta_n=\Delta_{-n-3}$ and it is known that 
$D_L=\mathfrak{f}_L^{p-1}$ for a cyclic number field $L$ of 
prime degree $p$. 
\\
\quad
Let $K/\Q$ be a finite Galois extension with the Galois group G.
We denote the ring of integers of $K$ by $\calO_K$.
If the set $\{\sigma(\alpha)\mid\sigma\in G\}$ is a basis of $\calO_K$ 
over $\Z$, we call the basis $\it{a\ normal\ integral\ basis}$
(abbreviated to NIB) and $\alpha$ the generator.
It is well-known that if $K/\mathbb{Q}$ has an NIB, then the extension is 
tamely ramified. 
For abelian number fields, we have the following.
\begin{thm}[Hilbert-Speiser]\label{thm:HS}
Let $K$ be an abelian number field.
The following three conditions are equivalent. 
\begin{itemize}
\item[(i)]
$K/\Q$ is tamely ramified.
\item[(ii)]
The conductor of $K$ is square-free. 
\item[(iii)]
$K/\Q$ has a normal integral basis.
\end{itemize}
\end{thm} 
\begin{proof}
See \cite[chap.\,9, Theorem 3.4]{L}.
\end{proof}
From this theorem and (\ref{eq:cond}), we know that $L_n$ has an NIB if and only if 
$3\nmid n$ or $n\equiv12 \pmod{27}$.
There is a one-to-one correspondence between
the generators of NIBs of a  finite Galois extension which has an NIB.
and  the  elements of the multiplicative  group $\mathbb Z [G]^{\times}$ of the group ring 
    of the Galois group $G$ (see Lemma\,\ref{lem:all NIB}). If $G=\mathrm{Gal}(L_n/\mathbb Q)=\langle \sigma \rangle$, then  we have 
$
\mathbb{Z}[G]^\times=\{\pm1_G,\pm\sigma,\pm\sigma^2\},
$
(\cite[p.\,2932 and 2933]{Ar}). Therefore, we know that 
there are 6 generators of NIBs  (two pairs of conjugate elements) and hence  $L_n/\mathbb Q$ has two normal integral bases.

In this paper, we first give all normal integral bases by the roots of Shanks' cubic polynomial for the simplest cubic field $L_n$ 
when they exist, that is, $L_n/\mathbb Q$ is tamely ramified (Theorems\,\ref{theo} and \ref{theo:main}).
 We see from Lemma\,\ref{lem:a0a1} that there are exactly $6$ pairs of integers $\{a_0, a_1\}$  
 that satisfy the assumptions of the theorems.
Next in this paper, as an application of Theorems\,\ref{theo} and \ref{theo:main} , we give an explicit relation  between the root of the Shanks' cubic polynomial $f_n (X)$ and the Gaussian period of $L_n$ 
in the case $L_n/\mathbb Q$ is tamely ramified.
In general, for a positive integer $\mathfrak f$, let $K=\mathbb Q(\zeta_{\mathfrak f})$ be the cyclotomic field where $\zeta_{\mathfrak f}$ is a primitive $\mathfrak f$-th root of unity
and $L$ be  an intermediate field of $K/\mathbb Q$.
The Gaussian periods of $L$ are defined by $\eta_0:=\mathrm{Tr}_{K/L}(\zeta_{\mathfrak f})$ and its conjugates.
When $\mathfrak f$ is square-free, then it is known that $\eta_0$ is a generator of a normal integral basis of $L/\mathbb Q$
(for example, see \cite[Proposition\,4.31]{N}).
For the simplest cubic field $L_n$ in the case $L_n/\mathbb Q$ is tamely ramified, the following has been known
by Lehmer \cite{Le} (when $\Delta_n =\mathfrak f_{L_n}$ is a prime number),  Ch\^{a}telet \cite{Ch} and Lazarus \cite{La}  (when  $\Delta_n =\mathfrak f_{L_n}$ is 
square-free).
If $\Delta_n =\mathfrak f_{L_n}$ and $\mathfrak f_{L_n}$ is square-free, then 
\begin{equation}\label{eq:Lehmer-gauss}
\pm \eta_0 = \rho_n + \frac{ \left(\dfrac{n}{3}\right)-n}{3 }
\end{equation}
 holds for $\rho_n$, 
which is one of the three roots of $f_n(X)$,
where $\left(\dfrac{n}{3}\right)$ is the Legendre symbol. 
Furthermore, the minimal polynomial of the Gaussian period $\eta_0$ is given by
\begin{equation}\label{eq:Lehmer-poly}
X^3-\mu (\Delta_n)X^2-\frac{\Delta_n -1}{3}X+\mu (\Delta_n )\frac{ (a+3)\Delta_n-1}{27},
\end{equation}
where $\mu$ is the M\"{o}bius function and $a=\pm (2n+3),\ a\equiv 1\pmod{3}$.
We generalize these results to general $L_n$ whose conductor $\mathfrak f_{L_n}$ is square-free (Theorem\,\ref{theo:MP-gauss}).
Finally, we give some numerical examples (Examples\,\ref{ex:286},\ref{ex:66}  and \ref{ex:table})
by using Magma and PARI/GP.
\section{Normal integral bases and the multiplicative group of a group ring}
\label{sect:NIB}
Let $K/\mathbb{Q}$ be a finite Galois extension with the Galois group $G$,
and $\mathcal{O}_K$ be the ring of integers of $K$.
The action of the group ring  $K[G]$ (or $\mathcal{O}_K[G]$) on $K$ is given by
\[
x.a:=\sum_{\sigma\in G}n_{\sigma}\sigma(a)\ \ 
\quad 
\left(x=\sum_{\sigma\in G}n_\sigma\sigma\in\ K[G]\ 
(\mathrm{or}\ \mathcal{O}_K[G]),\ a\in K\right),
\]
and $K$ (resp. $\mathcal{O}_K$) is a $K[G]$
(resp. $\mathcal{O}_K[G]$)-module with this action.
\\
\quad
It is known (\cite{E}) that there is a one-to-one correspondence between generators of NIBs of $K$ and elements of $\mathbb{Z}[G]^\times$,
but we give a proof for the reader's convenience.
\begin{lem}\label{lem:NB}
Assume that $\alpha\ (\in\mathcal{O}_K)$ is a generator of an NIB 
of $K$.
Let $x\in\mathbb{Z}[G]$.
We have $x.\alpha=0$ if and only if $x=0$. Namely,
\[
\mathrm{Ann}_{\mathbb{Z}[G]}(\alpha)
(:=\{x\in\mathbb{Z}[G]\mid x.\alpha=0\})=0.
\]  
Especially,
we have $x.\alpha=y.\alpha$ if and only if $x=y$ for $x,y\in\mathbb{Z}[G]$.
\end{lem}
\begin{proof}
The assertion follows immediately since the set
$\{\sigma(\alpha)\mid \sigma\in G\}$
is linearly independent over $\mathbb{Z}$.
\end{proof}
\begin{lem}\label{lem:2NIB}
Assume that both $\alpha_1\ (\in\mathcal{O}_K)$ and $\alpha_2\ (\in\mathcal{O}_K)$  are generators of NIBs of $K$.
We have $\alpha_1=u.\alpha_2$ for some $u\in\mathbb{Z}[G]^\times$. 
\end{lem}
\begin{proof}
From the assumption, there exist $v,w\in\mathbb{Z}[G]$ satisfying
$\alpha_2=v.\alpha_1$ and $\alpha_1=u.\alpha_2$.
Hence, we have $\alpha_1=u.\alpha_2=u.(v.\alpha_1)=(uv).\alpha_1$.
From Lemma $\ref{lem:NB}$, we have $uv=1$,
and similarly $vu=1$. We conclude that  $u\in\mathbb{Z}[G]^\times$.
\end{proof}
From Lemmas \ref{lem:NB} and \ref{lem:2NIB}, we have the following lemma.
\begin{lem}\label{lem:all NIB}
Assume that $\alpha\ (\in\mathcal{O}_K)$ is a generator of an NIB of $K$.
We have 
\[
\{x\in\calO_K \ |\ \text{a generator of an NIB of } K\}
=
\{u.\alpha\mid u\in\mathbb{Z}[G]^{\times}\},
\]
and there is a one-to-one correspondence between generators
of NIBs of $K$ and elements of $\mathbb{Z}[G]^\times$.
\end{lem}
If $G=\langle\sigma\rangle$ is a cyclic group of order $3$,
then we have
\[
\mathbb{Z}[G]^\times=\{\pm1_G,\pm\sigma,\pm\sigma^2\},
\]
(\cite[p.\,2932 and 2933]{Ar}).
\section{Integral bases of the simplest cubic fields}
\label{sect:IB}
Let $K$ be a cubic field.
We denote the conjugates of $\alpha\in K$ by $\alpha,\alpha',\alpha''$.
For $\alpha_1,\alpha_2,\alpha_3\in K$, put
\[
d(\alpha_1,\alpha_2,\alpha_3):=
\left|
\begin{array}{ccc}
\alpha_1 & \alpha_2 & \alpha_3
\\
\alpha_1' & \alpha_2' & \alpha_3'
\\
\alpha_1'' & \alpha_2'' & \alpha_3''
\end{array}
\right|^2.
\]
If $\{\alpha_1,\alpha_2,\alpha_3\}$ is an integral basis of $K$,
then $d(\alpha_1,\alpha_2,\alpha_3)$ is the $discriminant$ of $K$ and
denote it by $D_K=d(\alpha_1,\alpha_2,\alpha_3)$.
\\
\quad
Let $\rho_n$ be a root of Shanks' cubic polynomial $f_n(X)$.
If $\Delta_n=n^2+3n+9$ is a prime number, then $\{1,\rho_n,\rho_n^2\}$ is
an integral basis and hence $D_{L_n}=d(1,\rho_n,\rho_n^2)$.
In this section, we give an integral basis of the simplest cubic field
$L_n$ in the case that $L_n/\mathbb{Q}$ is tamely ramified
(that is, $3\nmid n$ or $n\equiv12 \pmod{27}$).
We write $\Delta_n=n^2+3n+9=bc^3$ with $b,c\in\mathbb{Z}_{>0}$ where $b$
is cube-free, and furthermore $b=de^2$ with $d,e\in\mathbb{Z}_{>0}$ where
$d$ and $e$ are square-free and $(d,e)=1$.
We use the well-known algorithm for finding an integral basis for cubic fields 
(see \cite{Al}, \cite{D}, \cite{H} and \cite{V}).
\begin{lem}\label{lem:H}
{\cite[p.\,15, Theorem\,1.5]{H}}
Let $g(X)=X^3+b_1X^2+c_1X+d_1$ be an irreducible polynomial in
$\mathbb{Z}[X]$, and $\rho$ be a root of $g(X)$.
Put $K=\mathbb{Q}(\rho)$.
Suppose that $s,a,t\in\mathbb{Z}$ satisfy the following conditions
$(\mathrm{i})$ and $(\mathrm{ii})$.
\[
(\mathrm{i})\ 
d(1,\rho_n,\rho_n^2)=s^6a^2D_K,
\hspace{3mm}
\]
\begin{eqnarray*}
(\mathrm{ii})\ 
\dfrac{1}{2}g''(t)
\!\!\!&\equiv&\!\!\!
0 \pmod{s},
\\
\ \,
g'(t)
\!\!\!&\equiv&\!\!\!
0 \pmod{s^2a},
\\
\ \,\,
g(t)
\!\!\!&\equiv&\!\!\!
0 \pmod{s^3a^2},
\end{eqnarray*}
where $g'(t)\ (resp.\ g''(t))$ is the derivative
$(resp.\ the\ second\ derivative)$ of $g(t)$.
Put
\begin{eqnarray*}
\phi
\!\!\!&:=&\!\!\!
\frac{1}{s}(\rho-t),
\\
\psi
\!\!\!&:=&\!\!\!
\frac{1}{s^2a}(\rho^2+(t+b_1)\rho+t^2+b_1t+c_1).
\end{eqnarray*}
Then $\{1,\phi,\psi\}$ is an integral basis of $K$.
\end{lem}
From Lemma\,\ref{lem:H}, we can give an integral basis of the simplest
cubic field $L_n$ as follows.
\begin{prop}\label{prop:IB Ln}
Let $n$ be an integer with $3\nmid n$ or $n\equiv12 \pmod{27}$
and $\Delta_n=n^2+3n+9=de^2c^3$ with $d,e,c\in\mathbb{Z}_{>0}$
where $d$ and $e$ are square-free and $(d,e)=1$.
Put
\[
t:=
\begin{cases}
un, & \mathrm{if}\ 3\nmid n, 
\\ 
\frac{n}{3}, & \mathrm{if}\ n\equiv12 \pmod{27},
\end{cases}
\]
where $u$ is an integer satisfying $3u\equiv1 \pmod{e^2c^3}$.
Put
\begin{eqnarray*}
\phi
\!\!\!&:=&\!\!\!
\frac{1}{c}(\rho_n-t),
\vspace{1mm}
\\
\psi
\!\!\!&:=&\!\!\!
\frac{1}{c^2e}(\rho_n^2+(t-n)\rho_n+t^2-nt-n-3).
\end{eqnarray*}
Then $\{1,\phi,\psi\}$ is an integral basis of $L_n=\mathbb{Q}(\rho_n)$.
\end{prop}
\begin{proof}
Since $3\nmid n$ or $n\equiv12 \pmod{27}$,
we have $D_{L_n}=d^2e^2$ 
from (\ref{eq:Df}) and (\ref{eq:cond}), and hence
\[
d(1,\rho_n,\rho_n^2)=\Delta_n^2=d^2e^4c^6=e^2c^6D_{L_n}.
\]
Therefore, it is enough to show the condition (ii) :
\begin{equation}\label{eq:f-cong}
\begin{split}
\frac{1}{2}f_n''(t) &\equiv0 \pmod{c},
\\
f_n'(t) &\equiv0 \pmod{c^2e},
\\
f_n(t) &\equiv0 \pmod{c^3e^2},
\end{split}
\end{equation}
in Lemma\,\ref{lem:H}.
\\
\quad
First, we have
\begin{equation}\label{eq:Delta}
\Delta_n=n^2+3n+9\equiv0 \pmod{c^3e^2},
\end{equation}
\begin{equation}\label{eq:n^3}
n^3=(n-3)(n^2+3n+9)+27\equiv27 \pmod{c^3e^2}.
\end{equation}
\quad
Suppose that $3\nmid n$.
From (\ref{eq:Delta}), (\ref{eq:n^3}) and $3u\equiv1 \pmod{c^3e^2}$,
we can show the congruences 
(\ref{eq:f-cong}) for $t=un$.
\\
\quad
Suppose that $n\equiv12 \pmod{27}$.
We have $3^3 || \Delta_n$ and hence $3||c$.
Let $t=n/3$.
From $n^2=\Delta_n-3n-9=de^2c^3-3n-9$
and $n^3=(n-3)\Delta_n+27=de^2c^3(n-3)+27$,
we have 
\begin{eqnarray*}
\frac{1}{2}f_n''(t)
\!\!\!&=&\!\!\!
3t-n=0,
\\
f_n'(t)
\!\!\!&=&\!\!\!
-\frac{1}{3}n^2-n-3=-\frac{1}{3}de^2c^3\equiv0\ 
\pmod{c^2e},
\\
f_n(t)
\!\!\!&=&\!\!\!
-\frac{2}{27}n^3-\frac{1}{3}n^2-n-1=-\frac{1}{27}de^2c^3(2n+3)\equiv0\ 
\pmod{c^3e^2}.
\end{eqnarray*}
For the last congruence of $f_n''(t)$,
we use $2n+3\equiv0 \pmod{27}$
since $n\equiv12 \pmod{27}$.
\end{proof}
\begin{rem}\label{rem:Gr-KS}
The cyclic cubic fields with a power integral basis are determined by Gras\,\cite{G}. 
Kashio and Sekigawa\,\cite{KS} simplified the necessary and sufficient conditions for having a power 
integral basis and gave its generator explicitly.
\end{rem}
\section{Main Theorems}\label{sect:main}

Let $n$ be an integer and $f_n(X)$ be the Shanks' cubic polynomial defined in
\S\ref{sect:intro}, 
$\rho_n\in\mathbb{C}$  a root of $f_n(X)$. Put $L_n:=\mathbb{Q}(\rho_n)$
and $G:={\rm Gal}(L_n/\mathbb Q)=\langle \sigma \rangle$.
We denote the conjugates of $\rho_n$ by $\rho_n,\ \rho'_n:=\sigma (\rho_n)$\ and $\rho''_n:=\sigma^2 (\rho_n)$.
In this section, we give all normal integral bases 
by the roots of $f_n(X)$ for the simplest cubic
fields $L_n$ which are tamely ramified over $\mathbb{Q}$.
We construct the generator of an $\mathrm{NIB}$ from a normal basis and the
integral basis given by Proposition\,\ref{prop:IB Ln}.
The idea is based on the method in \cite{Ac}.
Let $\zeta:=\zeta_3=e^{2\pi i/3}$ be a primitive cube root of unity,
and put $A_n:=n+3(1+\zeta)=n-3\zeta^2$.
We have $A_nA_n'=\Delta_n$, where $A_n'=n+3(1+\zeta^2)=n-3\zeta$ is the
conjugate of $A_n$.
\begin{lem}\label{lem:p01}
If $p$ is a prime number satisfying $p\mid\Delta_n$,
then we have $p\equiv0,1 \pmod{3}$.
\end{lem}
\begin{proof}
Assume that $p\mid\Delta_n$ and $p\equiv-1 \pmod{3}$.
Since $p$ is inert in $\mathbb{Q}(\zeta)$,
we have $p\mid A_n$ and $p\mid A_n'$.
However, since $\{1,\zeta\}$ is a basis of $\mathbb{Z}[\zeta]$ over
$\mathbb{Z}$ and $p\mid A_n=(n+3)+3\zeta$,
we have $p\mid n+3$ and $p\mid 3$.
This is a contradiction.
\end{proof}
For the roots $\rho_n,\rho_n',\rho_n''$ of Shanks' cubic polynomial
$f_n(X)$, we have $\rho_n'=-1/(1+\rho_n)$ and $\rho_n''=-1/(1+\rho_n')$
(\cite[p.\,1138]{S}).
We have the following.
\\
\begin{align}
& \rho_n^2\rho_n'+\rho_n'^2\rho_n''+\rho_n''^2\rho_n=3, \label{eq:roots1} \\
& \rho_n^2\rho_n''+\rho_n'^2\rho_n+\rho_n''^2\rho_n'=-n^2-3n-6, \label{eq:roots2} \\
& \rho_n^2\rho_n'+\rho_n'^2\rho_n''+\rho_n''^2\rho_n-\rho_n^2\rho_n''-\rho_n'^2\rho_n-\rho_n''^2\rho_n'=\Delta_n, \label{eq:roots3} \\
& \rho_n'=\rho_n^2-(n+1)\rho_n-2, \label{eq:roots4} \\
& \rho_n''=-\rho_n^2+n\rho_n+(n+2). \label{eq:roots5} 
\end{align}
\begin{lem}\label{lem:3eta}
Let $n\in\mathbb{Z}$ and $r_1,r_2,r_3\in\mathbb{Q}$.
For $\eta=r_1\rho_n+r_2\rho_n'+r_3$,
we have the following.
\\
\begin{itemize}
\item[(1)]
$\eta+\eta'+\eta''=n(r_1+r_2)+3r_3$.
\\
\item[(2)]
$\eta\eta'+\eta'\eta''+\eta''\eta
=r_1r_2n^2+2(r_1+r_2)r_3n-(r_1^2-r_1r_2+r_2^2)(n+3)+3r_3^2$.
\\
\item[(3)]
$\eta\eta'\eta''
=r_1r_2r_3n^2-r_1^2r_2n(n+3)+(r_1+r_2)r_3^2n-(r_1^2-r_1r_2+r_2^2)r_3(n+3)
\vspace{1.5mm}
\\
\hspace{13mm}
+r_1r_2(r_1-r_2)(n^2+3n+6)+r_1^3+r_2^3+r_3^3-3r_1^2r_2$.
\end{itemize}
\end{lem}
\begin{proof}
We can show the assertions by direct calculation using
$\rho_n+\rho_n'+\rho_n''=n,\ 
\rho_n\rho_n'+\rho_n'\rho_n''+\rho_n''\rho_n=-(n+3),\ 
\rho_n\rho_n'\rho_n''=1$ 
and (\ref{eq:roots2}).
\end{proof}
\begin{lem}\label{lem:a0a1}
Let $s$ be a positive integer satisfying $s\mid\Delta_n=A_nA_n'$.
There exist exactly $6$ pairs $\{a_0,a_1\}$ of integers satisfying
$s=a_0^2-a_0a_1+a_1^2$ and $a_0+a_1\zeta\mid A_n$ in $\mathbb{Z}[\zeta]$.
Furthermore, if $\{a_0,a_1\}$ is such a pair, then all the pair are
$\{\pm a_0,\pm a_1\},\ \{\pm a_1,\mp(a_0-a_1)\}$ and
$\{\pm(a_0-a_1),\pm a_0\}$, where the double-signs are in same order.
\end{lem}
\begin{proof}
From Lemma\,\ref{lem:p01}, we have $s=3^jp_1\cdots p_k$ with
$j\in\mathbb{Z}_{\geq0}$ and $p_1,\cdots,p_k$ are prime numbers with
$p_1\equiv\cdots\equiv p_k\equiv1 \pmod{3}$.
Since $\mathbb{Z}[\zeta]$ is a unique factorization domain and
$p_1,\cdots,p_k$ split in $\mathbb{Q}(\zeta)$, we can write
$p_i=\pi_i\pi_i'$ where $\pi_i$ and $\pi_i'$ are  prime elements with
$\pi_i\mid A_n$ and $\pi_i'\mid A_n'$, and the factorization is unique
up to elements of $\mathbb{Z}[\zeta]^\times=\{\pm1,\pm\zeta,\pm\zeta^2\}$.
Put $(1-\zeta)^j\pi_1\cdots\pi_k=a_0+a_1\zeta$ with $a_0,a_1\in\mathbb{Z}$.
We have 
\[
s
=3^j(\pi_1\cdots\pi_k)(\pi_1'\cdots\pi_k')
=(a_0+a_1\zeta)(a_0+a_1\zeta^2)
=a_0^2-a_0a_1+a_1^2,
\]
and $a_0+a_1\zeta\mid A_n$.
Furthermore, $a_0+a_1\zeta$ is unique up to elements of
$\mathbb{Z}[\zeta]^\times$, and all the elements are given by
$\pm(a_0+a_1\zeta),\ \pm\zeta(a_0+a_1\zeta),\ \pm\zeta^2(a_0+a_1\zeta)$.
\end{proof}
\begin{theo}\label{theo}
Let $n$ be an integer with $3\nmid n$ or $n\equiv12 \pmod{27}$,
$A_n=n+3(1+\zeta)$ and $\Delta_n=n^2+3n+9=de^2c^3$ with
$d,e,c\in\mathbb{Z}_{>0}$ where $d$ and $e$ are square-free and
$(d,e)=1$. 
Let $a_0$ and $a_1$ be integers satisfying $ec=a_0^2-a_0a_1+a_1^2$ and
$a_0+a_1\zeta\mid A_n$.
Put $m=(\varepsilon ec^2-n(a_0+a_1))/3\in\mathbb{Z}$ where
$\varepsilon\in\{\pm1\}$ is given by
\[
\varepsilon\equiv
\begin{cases}
n(a_0+a_1) & (\mathrm{mod}\ 3),\ \ \mathrm{if}\ 3\nmid n,
\\
a_0 & (\mathrm{mod}\ 3),\ \ \mathrm{if}\ n\equiv12 \pmod{27}.
\end{cases}
\]
Then
\begin{equation*}
\frac{1}{ec^2}(a_0\rho_n+a_1\rho_n'+m)
=\frac{1}{ec^2}(a_1\rho_n^2+(a_0-a_1n-a_1)\rho_n+m-2a_1)
\end{equation*}
is a generator of a normal integral basis of the simplest cubic field $L_n$.
\end{theo}
\begin{proof}
Since
\[
d(\rho_n,\rho_n',\rho_n'')
=(\rho_n+\rho_n'+\rho_n'')^2
(3(\rho_n\rho_n'+\rho_n'\rho_n''+\rho_n''\rho_n)
-(\rho_n+\rho_n'+\rho_n'')^2)^2=n^2\Delta_n^2\neq0,
\]
it follows that $\{\rho_n,\rho_n',\rho_n''\}$
is a normal  basis of $L_n$.
Let $\{1,\phi,\psi \}$ be the integral basis of
$L_n$ given by Proposition~\ref{prop:IB Ln}.
Using $(\ref{eq:roots4})$, we obtain
\begin{equation}\label{eq:1-phi-psi}
\begin{split}
1&=\frac{1}{n}(\rho_n+\rho_n'+\rho_n''),
\\
\phi &=\frac{1}{c}(\rho_n-t)
=\frac{1}{c}\left(\rho_n-\frac{t}{n}(\rho_n+\rho_n'+\rho_n'')\right)
=\frac{1}{cn}((n-t)\rho_n-t\rho_n'-t\rho_n''),
\\
\psi &=\frac{1}{ec^2}(\rho_n^2+(t-n)\rho_n+t^2-nt-n-3) 
\\
&=\frac{1}{ec^2}(\rho_n'+(t+1)\rho_n+t^2-nt-n-1) 
\\
& =\frac{1}{ec^2}\left(\rho_n'+(t+1)\rho_n
+\frac{t^2-nt-n-1}{n}(\rho_n+\rho_n'+\rho_n'')\right) 
\\
& =\frac{1}{ec^2n}((t^2-1)\rho_n+(t^2-nt-1)\rho_n'+(t^2-nt-n-1)\rho_n'').
\end{split}
\end{equation}
\\
Let $\ell :=ec^2n$.
Since $\{1,\phi,\psi\}$ is an integral basis,
we have from (\ref{eq:1-phi-psi})
\begin{equation}\label{eq:inc}
\ell \mathcal{O}_{L_n}\subset\mathbb{Z}[G].\rho_n.
\end{equation}
Let $\alpha$ be a generator of an $\mathrm{NIB}$ of $L_n$.
From (\ref{eq:inc}), there exists $g\in\mathbb{Z}[G]$ satisfying
$\alpha=g.(\rho_n/\ell )$. 
Therefore, we have
\begin{equation}\label{eq:=}
\mathcal{O}_{L_n}
=\mathbb{Z}[G].\alpha
=g\mathbb{Z}[G]. \frac{\rho_n}{\ell}. 
\end{equation}
Furthermore, the equalities of (\ref{eq:1-phi-psi}) can be written as
\begin{equation}\label{eq:1-phi-psi-2}
\begin{split}
1 &=g_1.\frac{\rho_n}{\ell},\ \ g_1:=\frac{\ell}{n}(1+\sigma+\sigma^2),
\\
\phi &=g_2.\frac{\rho_n}{\ell},\ \ 
g_2:=\frac{\ell}{cn}((n-t)-t\sigma-t\sigma^2), 
\\
\psi &=g_3.\frac{\rho_n}{\ell},\ \ 
g_3:=\frac{\ell}{ec^2n}((t^2-1)+(t^2-nt-1)\sigma+(t^2-nt-n-1)\sigma^2).
\end{split}
\end{equation}
\\
From (\ref{eq:=}), (\ref{eq:1-phi-psi-2}) and 
\[
\mathcal{O}_{L_n}=\mathbb{Z}+\phi\mathbb{Z}+\psi\mathbb{Z}
=(g_1\mathbb{Z}[G]+g_2\mathbb{Z}[G]+g_3\mathbb{Z}[G]).\frac{\rho_n}{l},
\]
we obtain the equality as ideals of $\mathbb{Z}[G]$ :
\[
(g)+\mathrm{Ann}_{\mathbb{Z}[G]}\left(\frac{\rho_n}{\ell}\right)
=(g_1,g_2,g_3)+\mathrm{Ann}_{\mathbb{Z}[G]}\left(\frac{\rho_n}{\ell}\right).
\]
Since $\{\rho_n/\ell,\rho_n'/\ell,\rho_n''/\ell\}$ is a normal basis of $L_n$,
we have $\mathrm{Ann}_{\mathbb{Z}[G]}\left(\rho_n/\ell\right)=0$,
and hence we get the equality as ideals of $\mathbb{Z}[G]$ :
\[
(g)=(g_1,g_2,g_3).
\]
Consider the surjective ring homomorphism
\[
\nu\ :\ \mathbb{Z}[G]\longrightarrow\mathbb{Z}[\zeta],
\]
defined by $\nu(\sigma)=\zeta$.
We calculate the image of the ideal $I:=(g)=(g_1,g_2,g_3)$ of
$\mathbb{Z}[G]$ by $\nu$.
Since $\nu$ is surjective, we obtain the ideal of $\mathbb{Z}[\zeta]$ :
\begin{equation}\label{eq:nu}
\nu(I)=(\nu(g))=(\nu(g_1),\nu(g_2),\nu(g_3)).
\end{equation}
From (\ref{eq:1-phi-psi-2}), the elements $\nu(g_1),\nu(g_2)$ and $\nu(g_3)$ are given by
\begin{eqnarray*}
\nu(g_1) 
\!\!\!&=&\!\!\! 
\frac{\ell}{n}(1+\zeta+\zeta^2)=0,
\\
\\
\nu(g_2) 
\!\!\!&=&\!\!\!
\frac{\ell}{cn}((n-t)-t\zeta-t\zeta^2)=\frac{\ell}{c}=ecn,
\\
\\
\nu(g_3) 
\!\!\!&=&\!\!\! 
\frac{\ell}{ec^2n}((t^2-1)+(t^2-nt-1)\zeta+(t^2-nt-n-1)\zeta^2)
\\
\\
\!\!\!&=&\!\!\! 
-nt\zeta-n(t+1)\zeta^2
=
n(t+1+\zeta).
\end{eqnarray*}
Therefore, we have
\[
\nu(I)=n(ec,t+1+\zeta).
\]
\quad
Suppose that $3\nmid n$.
Since $t=un$ and $3u\equiv1 \pmod{e^2c^3}$,
we have $\nu(I)=n(ec,uA_n)$ with $A_n=n+3(1+\zeta)$.
From Lemma\,\ref{lem:p01} and $3\nmid ec$,
it follows that any prime number $p$ dividing $ec$ satisfies
$p\equiv1 \pmod{3}$ and hence we have
\[
ec=(\pi_1\cdots\pi_k)(\pi_1'\cdots\pi_k'),
\]
where $\pi_1,\cdots,\pi_k$ are prime elements dividing $A_n$ and
$\pi_1',\cdots,\pi_k'$ are their conjugates respectively.
Furthermore, since $(ce,u)=1$, we conclude that
\begin{equation}\label{eq:nu-gamma1}
\begin{split}
\nu(I)
&=n\pi_1\cdots\pi_k(\pi_1'\cdots\pi_k',uA_n(\pi_1\cdots\pi_k)^{-1})
\\
&=n\pi_1\cdots\pi_k(1)
\\
&=(\gamma),
\end{split}
\end{equation}
where $\gamma:=n\pi_1\cdots\pi_k=n(a_0+a_1\zeta)$ and
$a_0,a_1\in\mathbb{Z}$ are integers satisfying $ec=a_0^2-a_0a_1+a_1^2$
and $a_0+a_1\zeta\mid A_n$.
\\
\quad
Suppose that $n\equiv12 \pmod{27}$.
Since $t=n/3$,
we have $\nu(I)=n/3(3ec,A_n)$ with $A_n=n+3(1+\zeta)$.
From Lemma\,\ref{lem:p01} and $3\mid\mid ec$,
it follows that any prime number $p$ dividing $ec$ satisfies $p=3$ or
$p\equiv1 \pmod{3}$ and hence we have
\[
ec=3(\pi_1\cdots\pi_k)(\pi_1'\cdots\pi_k'),
\]
where $\pi_1,\cdots,\pi_k$ are prime elements dividing $A_n$ and
$\pi_1',\cdots,\pi_k'$ are their conjugates respectively.
Furthermore, since $3^3\mid\mid\Delta_n=A_nA_n'$, we conclude that
\begin{equation}\label{eq:nu-gamma2}
\begin{split}
\nu(I)
&=n(1-\zeta)\pi_1\cdots\pi_k
((1-\zeta)\pi_1'\cdots\pi_k',A_n(3(1-\zeta)\pi_1\cdots\pi_k)^{-1})
\\
&=n(1-\zeta)\pi_1\cdots\pi_k(1)
\\
&=(\gamma),
\end{split}
\end{equation}
where $\gamma:=n(1-\zeta)\pi_1\cdots\pi_k=n(a_0+a_1\zeta)$ and
$a_0,a_1\in\mathbb{Z}$ are integers satisfying $ec=a_0^2-a_0a_1+a_1^2$
and $a_0+a_1\zeta\mid A_n$.
\\
\quad
From (\ref{eq:nu-gamma1}) and (\ref{eq:nu-gamma2}), we have the following.
\begin{equation}\label{eq:nu=gamma}
\nu(I)=(\gamma),\ \ \gamma:=n(a_0+a_1\zeta),
\end{equation}
where $a_0,a_1\in\mathbb{Z}$ are integers satisfying
$ec=a_0^2-a_0a_1+a_1^2$ and $a_0+a_1\zeta\mid A_n$.
\\
\quad
Put $x:=n(a_0+a_1\sigma)\in\mathbb{Z}[G]$.
From (\ref{eq:nu}) and (\ref{eq:nu=gamma}),
we have $\nu(g)=v\gamma$ with $v\in\mathbb{Z}[\zeta]^\times$.
Since $\mathbb{Z}[\zeta]^\times=\{\pm1,\pm\zeta,\pm\zeta^2\}$ and
$\mathbb{Z}[G]^\times=\{\pm1_G,\pm\sigma,\pm\sigma^2\}$,
there exists $\xi\in\mathbb{Z}[G]^\times$ satisfying
$\nu(\xi)=v^{-1}$, and we have $\nu(\xi g)=\gamma=\nu(x)$.
We conclude that $\xi g-x\in\mathrm{Ker}~\nu=(1_G+\sigma+\sigma^2)$.
Therefore, there exists $m\in\mathbb{Z}$ satisfying
\[
(\xi g-x).\frac{\rho_n}{\ell}=m\mathrm{Tr}\left(\frac{\rho_n}{\ell}\right),
\]
where $\mathrm{Tr}$ is the trace map from $L_n$ to $\mathbb{Q}$.
We have 
\begin{equation}\label{eq:xi}
\begin{split}
\xi.\alpha &=(\xi g).\frac{\rho_n}{\ell}
\\
&=x.\frac{\rho_n}{\ell}+m\mathrm{Tr}\left(\frac{\rho_n}{\ell}\right)
\\
&=\frac{1}{ec^2}(a_0\rho_n+a_1\rho_n'+m).
\end{split}
\end{equation}
\\
Since $\alpha$ is a generator of an NIB and
$\xi\in\mathbb{Z}[G]^\times$, it follows from Lemma\,\ref{lem:all NIB} that
$\xi.\alpha$ is also a generator of an NIB. 
Therefore, a generator of an NIB  of $L_n$ is given by
\begin{equation}\label{eq:NIB}
\frac{1}{ec^2}(a_0\rho_n+a_1\rho_n'+m)
\end{equation}
from (\ref{eq:xi}).
If $\beta$ is a generator of an NIB of $L_n$,
then we have
$(\beta+\beta'+\beta'')/\mathrm{Tr}(\beta)=1=y.\beta$ for
$y\in\mathbb{Z}[G]$.
Therefore, we have $1/\mathrm{Tr}(\beta)\in\mathbb{Z}$ and hence
$\mathrm{Tr}(\beta)=\pm1$.
The trace of (\ref{eq:NIB}) is
\begin{equation}\label{eq:tr}
\e
=
\frac{1}{ec^2}\mathrm{Tr}(a_0\rho_n+a_1\rho_n'+m)
=
\frac{1}{ec^2}(n(a_0+a_1)+3m),
\end{equation}
for $\e\in\{\pm1\}$, and hence we have
\begin{equation}\label{eq:m}
m
=
\frac{\e ec^2-n(a_0+a_1)}{3}.
\end{equation}
\quad
Next, we determine the sign $\e\in\{\pm1\}$.
Suppose that $3\nmid n$.
From Lemma\,\ref{lem:p01} and $\Delta_n=dec^3$,
we have $e\equiv c\equiv1 \pmod{3}$.
We conclude that $3\nmid a_0+a_1$ and  $\e\equiv n(a_0+a_1) \pmod{3}$ from (\ref{eq:tr}).
\\
\quad
Suppose that $n\equiv12 \pmod{27}$.
Since $3\mid\mid c$ and
$ec=a_0^2-a_0a_1+a_1^2=(a_0+a_1\zeta)(a_0+a_1\zeta^2)$,
we have 
\begin{equation}\label{eq:||}
1-\zeta\mid\mid a_0+a_1\zeta,\ \ 
1-\zeta\mid\mid a_0+a_1\zeta^2.
\end{equation}
If $3\mid a_0$,
then we have $a_1\equiv a_0+a_1\zeta\equiv0  \pmod{1-\zeta}$,
and hence $3\mid a_1$.
This is a contradiction to (\ref{eq:||}).
Therefore, we have $3\nmid a_0$, and similarly $3\nmid a_1$.
Put $\eta:=(a_0\rho_n+a_1\rho_n'+m)/ec^2$.
From Lemma\,\ref{lem:3eta}, we have
\[
\eta\eta'\eta''=\frac{1}{e^3c^6}R
\]
where
\begin{equation}\label{eq:R}
\begin{split}
R&:=(a_0+a_1)m^2n-a_0a_1^2(n^2+3n+9)+a_0a_1mn^2
\\
& \quad -(a_0^2-a_0a_1+a_1^2)m(n+3)+m^3+(a_0+a_1)^3.
\end{split}
\end{equation}
Since $\eta\in\mathcal{O}_{L_n}$, we have $\eta\eta'\eta''\in\mathbb{Z}$.
Furthermore, since $3\mid\mid c$ and $3\nmid e$,
we have $3^6\mid R$.
From (\ref{eq:||}), we have $a_0+a_1\equiv a_0+a_1\zeta\equiv0 \pmod{1-\zeta}$
and hence $3\mid a_0+a_1$.
Furthermore, from (\ref{eq:tr}), we have $\e ec^2=n(a_0+a_1)+3m$.
From this equality, we obtain $3\mid m$.
From (\ref{eq:R}), $n^2+3n+9=\Delta_n=de^2c^3$ and $a_0^2-a_0a_1+a_1^2=ec$,
we have
\[
-a_0a_1^2de^2c^3+a_0a_1mn^2-ecm(n+3)+m^3+(a_0+a_1)^3\equiv0 \pmod{3^4}.
\]
Put $C:=c/3,\ M:=m/3$ and $N:=n/3\in\mathbb{Z}$.
Dividing the above congruence by $3^3$, we obtain
\begin{equation}\label{eq:mod3-1}
-a_0a_1^2de^2C^3+a_0a_1MN^2-eCM(N+1)+M^3+\left(\frac{a_0+a_1}{3}\right)^3
\equiv0 \pmod{3}.
\end{equation}
From Lemma\,\ref{lem:p01} and $\Delta_n=de^2c^3$,
we have $d\equiv e\equiv C\equiv1 \pmod{3}$.
Furthermore, we have $a_1^2\equiv1 \pmod{3}$ since $3\nmid a_1$,
and $N\equiv1\ \mod{3}$ since $n\equiv12 \pmod{27}$.
From (\ref{eq:mod3-1}), we have
\begin{equation}\label{eq:mod3-2}
-a_0+a_0a_1M-2M+M^3+\left(\frac{a_0+a_1}{3}\right)^3\equiv0 \pmod{3}.
\end{equation}
Dividing the equality $\e ec^2=n(a_0+a_1)+3m$ by $3^2$,
we obtain $\e eC^2=(a_0+a_1)N/3+M$,
and hence $(a_0+a_1)/3\equiv\e-M \pmod{3}$.
From this congruence and (\ref{eq:mod3-2}), we have
\begin{equation}\label{eq:mod3-3}
-a_0+(a_0a_1+1)M+\e\equiv0 \pmod{3}.
\end{equation}
If $a_0\equiv a_1 \pmod{3}$ then we have
$a_0+a_1\zeta\equiv a_0(1+\zeta)\not\equiv0 \pmod{1-\zeta}$.
This is a contradiction to (\ref{eq:||}).
Therefore, we have 
$a_0\equiv -a_1 \pmod{3}$,
and hence $a_0a_1\equiv-1 \pmod{3}$.
From (\ref{eq:mod3-3}), we obtain $\e\equiv a_0 \pmod{3}$.
\\
\quad
Finally, the last equality of the theorem follows from (\ref{eq:roots4}).
The proof is complete.
\end{proof}
\begin{cor}\label{cor:NIB,iff}
Let $n\in\mathbb{Z}$ satisfying $3\nmid n$. 
The simplest cubic field $L_n$ has a normal integral basis of the form
\[
\{v+w\rho_n,v+w\rho'_n,v+w\rho''_n\}\ \ (v,w\in\mathbb{Z})
\]
if and only if $\Delta_n$ is square-free.
Furthermore in this case,
the integers $v$ and $w$ are given by
\[
w=\pm1,\ and\ 
v=\dfrac{1}{3}w\left(\left(\dfrac{n}{3}\right)-n\right),
\]
where $\left(\dfrac{n}{3}\right)$ is the Legendre symbol.
\end{cor}
\begin{proof}
From Lemma\,\ref{lem:all NIB}, Theorem\,\ref{theo} and
$\mathbb{Z}[G]^\times=\{\pm1_G,\pm\sigma,\pm\sigma^2\}$,
all normal integral bases are given by the conjugates of
\[
\pm\frac{1}{ec^2}(a_1\rho_n^2+(a_0-a_1n-a_1)\rho_n+m-2a_1),
\]
where $a_0,a_1$ and $m$ are the integers defined in Theorem\,\ref{theo}.
Therefore, $L_n$ has a normal integral basis of the form
\[
\{v+w\rho_n,v+w\rho_n',v+w\rho_n''\}\ \ (v,w\in\mathbb{Z})
\]
if and only if $a_1=0$ and $a_0/ec^2 \in \mathbb Z$.
In this case, since $a_0+a_1\zeta=a_0$ divides
$A_n=(n+3)+3\zeta$, we have $a_0=\pm1,\pm3$.
However, since $ec=a_0^2-a_0a_1+a_1^2=a_0^2$ is not divisible by 3,
we conclude $a_0=\pm1$, and hence $ec=1$.
This means that $\Delta_n$ is square-free.
The generators of NIBs in this case are
$ \pm (a_0((\frac{n}{3})-n)/3+a_0\rho_n )$ with $a_0=\pm1$. 
\end{proof}
Next, we consider the case of $n\equiv12 \pmod{27}$.
In this case,
we have $3^3\mid\mid\Delta_n:=n^2+3n+9=bc^3$,
and hence $3\nmid b$ and $3\mid\mid c$.
\begin{cor}\label{cor:new case}
Let $n\in\mathbb{Z}$ satisfying $n\equiv12 \pmod{27}$.
Suppose that $\Delta_n/3^3$ is squarefree.
Then the set
\[
\left\{\frac{1}{9}(\rho_n-\rho_n'+3),
\frac{1}{9}(\rho_n'-\rho_n''+3),
\frac{1}{9}(\rho_n''-\rho_n+3)
\right\}
\]
is a normal integral basis of the simplest cubic field $L_n$.
\end{cor}
\begin{proof}
Since $\Delta_n/3^3=de^2c^3/3^3$ is square-free,
we have 
$e=1$ and $c=3$.
For $a_0=1$ and $a_1=-1$, we have $ec=3=a_0^2-a_0a_1+a_1^2$ and
$a_0+a_1\zeta\mid A_n$.
It follows from Theorem\,\ref{theo} that $(\rho_n-\rho_n'+3)/9$
is a generator of an NIB.
\end{proof}
We obtain the following theorem from Lemma\,\ref{lem:all NIB},
Theorem\,\ref{theo} and
$\mathbb{Z}[G]^\times=\{\pm1,\pm\sigma,\pm\sigma^2\}$ for
$G=\mathrm{Gal}(L_n/\mathbb{Q})=\langle\sigma\rangle$.
\begin{theo}\label{theo:main}
Let $n$ be an integer with $3\nmid n$ or $n\equiv12\pmod{27}$,
and $a_0,a_1,m$ be integers defined in Theorem\,$\mathrm{\ref{theo}}$.
Then all normal integral bases of the simplest cubic field $L_n$ are
\[
\left\{
\frac{1}{ec^2}(a_0\rho_n+a_1\rho_n'+m),
\frac{1}{ec^2}(a_0\rho'_n+a_1\rho_n''+m),
\frac{1}{ec^2}(a_0\rho''_n+a_1\rho_n+m)
\right\}
\]
and
\[
\left\{
-\frac{1}{ec^2}(a_0\rho_n+a_1\rho_n'+m),
-\frac{1}{ec^2}(a_0\rho'_n+a_1\rho_n''+m),
-\frac{1}{ec^2}(a_0\rho''_n+a_1\rho_n+m)
\right\}.
\]
\end{theo}
We obtain the following corollary from 
Corollaries \ref{cor:NIB,iff} and \ref{cor:new case} 
\begin{cor}\label{cor:3all NIB}
Let $n\in\mathbb{Z}$.
We have the following.
\begin{itemize}
\item[(1)]
If $3\nmid n$ and $\Delta_n$ is square-free,
then all normal integral bases of $L_n$ are
\[
\{v_n+\rho_n,v_n+\rho_n',v_n+\rho_n''\}\ and\ 
\{-v_n-\rho_n,-v_n-\rho_n',-v_n-\rho_n''\},
\]
where $v_n:=((\frac{n}{3})-n)/3$.
\\
\item[(2)]
If $n\equiv12 \pmod{27}$ and $\Delta_n/3^3$ is square-free, 
then all normal integral bases of $L_n$ are
\[
\left\{
\frac{1}{9}(\rho_n-\rho_n'+3),
\frac{1}{9}(\rho_n'-\rho_n''+3),
\frac{1}{9}(\rho_n''-\rho_n+3)
\right\}
\]
and
\[
\left\{
-\frac{1}{9}(\rho_n-\rho_n'+3),
-\frac{1}{9}(\rho_n'-\rho_n''+3),
-\frac{1}{9}(\rho_n''-\rho_n+3)
\right\}.
\]
\end{itemize}
\end{cor}
We obtain the following theorem from Lemma~\ref{lem:3eta} and Theorem~\ref{theo:main}.
\begin{theo}\label{m.p.}
Let $n$ be an integer with $3\nmid n$ or $n\equiv12\pmod{27}$,
and $a_0,a_1,m$ and $\e\in\{\pm1\}$ be integers defined in
Theorem\,$\mathrm{\ref{theo}}$.
Put $\alpha=(a_0\rho_n+a_1\rho_n'+m)/ec^2$.
Then the minimal polynomial of the generator $\pm\alpha$
of an NIB of the simplest cubic field $L_n$ is
\begin{eqnarray*}
F_\pm(X)=X^3
\!\!\!\!&\mp&\!\!\!\!
\e X^2
+
\frac{1}{e^2c^4}(a_0a_1n^2+2(a_0+a_1)mn-ec(n+3)+3m^2)X
\\
\!\!\!\!&\mp&\!\!\!\!
\frac{1}{e^3c^6}(a_0a_1mn^2-a_0^2a_1n(n+3)+(a_0+a_1)m^2n-ecm(n+3)
\\
\!\!\!\!&+&\!\!\!\!
a_0a_1(a_0-a_1)(n^2+3n+6)+a_0^3+a_1^3+m^3-3a_0^2a_1),
\end{eqnarray*}
where the double-signs are in same order.
\end{theo}
\begin{cor}\label{cor:f,g,h}
\begin{itemize}
\item[(1)]
Suppose that $n\equiv1 \pmod{3}$ and $\Delta_n$ is square-free.
Put $\alpha=(1-n)/3+\rho_n$.
The minimal polynomial of the generator $\pm\alpha$ of 
an NIB of $L_n$ is 
\[
f_\pm(X)
=
X^3\mp X^2-\frac{1}{3}(n^2+3n+8)X\mp\frac{1}{27}(2n^3+6n^2+18n+1),
\]
where the double-signs are in same order.
\\
\item[(2)] 
Suppose that $n\equiv2 \pmod{3}$ and $\Delta_n$ is square-free.
Put $\alpha=(1+n)/3-\rho_n$.
The minimal polynomial of the generator of $\pm\alpha$ of
an NIB of $L_n$ is
\[
g_\pm(X)
=
X^3\mp X^2-\frac{1}{3}(n^2+3n+8)X\pm\frac{1}{27}(2n^3+12n^2+36n+53),
\]
where the double-signs are in same order.
\\
\item[(3)] 
Suppose that $n\equiv12 \pmod{27}$ and
$\Delta_n/27$ is square-free.
Put $\alpha=(\rho_n-\rho_n'+3)/9$.
The minimal polynomial of the generator of $\pm\alpha$ of
an NIB  of $L_n$ is
\[
h_\pm(X)
=
X^3\mp X^2-\frac{1}{3^4}(n^2+3n-18)X\pm\frac{1}{3^6}(4n^2+12n+9),
\]
where the double-signs are in same order.
\end{itemize}
\end{cor}
\begin{proof}
\begin{itemize}
\item[(1)]
It is obtained by putting $e=c=1$, $a_0=1,\ a_1=0,\ \e=1$, and
$m=(1-n)/3$ in Theorem\,\ref{m.p.}.
\\
\item[(2)]
It is obtained by putting $e=c=1$, $a_0=-1,\ a_1=0,\ \e=1$ and
$m=(1+n)/3$ in Theorem\,\ref{m.p.}.
\\
\item[(3)]
It is obtained by putting $e=1$, $c=3$, $a_0=1,\ a_1=-1,\ \e=1$, and
$m=3$ in Theorem\,\ref{m.p.}.
\end{itemize}
\end{proof}
\begin{ex}\label{ex:286}
Let $n=286=2\cdot11\cdot13$.
We have $\Delta_n=241\cdot7^3,\ d=241,\ e=1,\ c=7$,
and $\mathfrak{f}_{L_n}=241,\ D_{L_n}=241^2$.
In this case, all the pairs $\{a_0,a_1\}$ satisfying
$ec=a_0^2-a_0a_1+a_1^2$ and $a_0+a_1\zeta\mid A_n$
are $\{\pm2,\pm3\},\ \{\pm3,\pm1\},\ \{\mp1,\pm2\}$
(the double-signs are in same order).
Therefore, all the generators of NIBs of $L_n$ and their minimal
polynomials are given in the following table.
\begin{table}[H]
  \caption{$n=286$}
\begin{center}
\begingroup
\renewcommand{\arraystretch}{1.5}
\begin{tabular}{|c||c||c|}
\hline
\hspace{2mm}
$\{a_0,a_1\}$
\hspace{2mm}
&
\hspace{9mm}
generator of $\mathrm{NIB}$
\hspace{9mm}
&
\hspace{5mm}
minimal polynomial 
\hspace{5mm}
\\
\hline
\hline
$\{2,3\}$
&
$\frac{1}{49}(3\rho_n^2-859\rho_n-499)$
&
$X^3+X^2-80X+125$
\\  
\hline
$\{-3,-1\}$
&
$-\frac{1}{49}(\rho_n^2-284\rho_n-367)$
&
$X^3+X^2-80X+125$
\\  
\hline
$\{1,-2\}$
&
$-\frac{1}{49}(2\rho_n^2-575\rho_n-83)$
&
$X^3+X^2-80X+125$
\\  
\hline
$\{-2,-3\}$
&
$-\frac{1}{49}(3\rho_n^2-859\rho_n-499)$
&
$X^3-X^2-80X-125$
\\  
\hline
$\{3,1\}$
&
$\frac{1}{49}(\rho_n^2-284\rho_n-367)$
&
$X^3-X^2-80X-125$
\\  
\hline
$\{-1,2\}$
&
$\frac{1}{49}(2\rho_n^2-575\rho_n-83)$
&
$X^3-X^2-80X-125$
\\
\hline
\end{tabular}
\endgroup
\end{center}
\end{table}
\end{ex}
%
\begin{ex}\label{ex:66}
Let $n=66=2\cdot3\cdot11\equiv12 \pmod{27}$.
We have $\Delta_n=13^2\cdot3^3,\ d=1,\ e=13,\ c=3$,
and $\mathfrak{f}_{L_n}=13,\ D_{L_n}=13^2$.
In this case, all the pairs $\{a_0,a_1\}$ satisfying
$ec=a_0^2-a_0a_1+a_1^2$ and $a_0+a_1\zeta\mid A_n$
are $\{\pm5,\pm7\},\ \{\pm7,\pm2\},\ \{\mp2,\pm5\}$
(the double-signs are in same order).
Therefore, all the generators of NIBs of $L_n$ and their minimal
polynomials are given in the following table.
%
%
%
%
%
\begin{table}[H]
  \caption{$n=66$}
\begin{center}
\begingroup
\renewcommand{\arraystretch}{1.5}
\begin{tabular}{|c||c||c|}
\hline
\hspace{2mm}
$\{a_0,a_1\}$
\hspace{2mm}
&
\hspace{9mm}
generator of $\mathrm{NIB}$
\hspace{9mm}
&
\hspace{5mm}
minimal polynomial 
\hspace{5mm}
\\
\hline
\hline
$\{5,7\}$
&
$\frac{1}{117}(7\rho_n^2-464\rho_n-317)$
&
$X^3+X^2-4X+1$
\\  
\hline
$\{-7,-2\}$
&
$-\frac{1}{117}(2\rho_n^2-127\rho_n-163)$
&
$X^3+X^2-4X+1$
\\  
\hline
$\{2,-5\}$
&
$-\frac{1}{117}(5\rho_n^2-337\rho_n-37)$
&
$X^3+X^2-4X+1$
\\  
\hline
$\{-5,-7\}$
&
$-\frac{1}{117}(7\rho_n^2-464\rho_n-317)$
&
$X^3-X^2-4X-1$
\\  
\hline
$\{7,2\}$
&
$\frac{1}{117}(2\rho_n^2-127\rho_n-163)$
&
$X^3-X^2-4X-1$
\\  
\hline
$\{-2,5\}$
&
$\frac{1}{117}(5\rho_n^2-337\rho_n-37)$
&
$X^3-X^2-4X-1$
\\
\hline
\end{tabular}
\endgroup
\end{center}
\end{table}
\end{ex}
%
\section{Gaussian periods and roots of Shanks' cubic polynomial}
\label{sect:gaussian}
Let $n$ be an integer and $L_n$ the simplest cubic field defined in
\S\ref{sect:intro} and put ${\rm Gal}(L_n/\mathbb Q)=\langle\sigma \rangle$.
\begin{df}[Gaussian periods]\label{df:GP}
Set $K_n=\mathbb{Q}(\zeta_{\mathfrak{f}_{L_n}})$
where $\zeta_{\mathfrak{f}_{L_n}}$ is a primitive $\mathfrak{f}_{L_n}$th
root of unity. 
We define {\it Gaussian periods} $\eta_i\ (i=0,1,2)$ as 
the conjugates of
$\eta_0:=\mathrm{Tr}_{K_n/L_n}(\zeta_{\mathfrak{f}_{L_n}})$.
\end{df}
From now on, we assume that $L_n/\mathbb Q$ is tamely ramified.
If $\Delta_n=\mathfrak{f}_{L_n}$,
then it is known the relation between roots of $f_n(X)$ and 
gaussian periods $\eta_i$ by Lehmer \cite[p. 536]{Le},
Ch\^{a}telet \cite{Ch} and Lazarus \cite[Proposition\,2.2]{La}. 
Hence, we consider the case where $\mathfrak{f}_{L_n}$
is not necessarily equal to $\Delta_n$.
Since $\eta_0=\mathrm{Tr}_{K_n/L_n}(\zeta_{\mathfrak{f}_{L_n}})$ is a
generator of an NIB for $L_n$ (see \cite[Prop\,4.31]{N}),
it is equal to a conjugate of $\pm\a$, where
\[
\a
=
\frac{1}{ec^2}(a_0\rho_n+a_1\rho_n'+m)
=
\frac{1}{ec^2}(a_1\rho_n^2+(a_0-a_1n-a_1)\rho_n+m-2a_1)
\]
is the generator of an NIB in Theorems\,1, 2 and \ref{m.p.}.
From Theorem\,\ref{m.p.}, we have $\mathrm{Tr}_{L_n/\mathbb{Q}}(\a)=\e$,
where $\e\in\{\pm1\}$ is defined in Theorem\,\ref{theo}.
On the other hand, for any $i\in\{0,1,2\}$ we have
\[
\mathrm{Tr}_{L_n/\mathbb{Q}}(\eta_i)
=
\mathrm{Tr}_{L_n/\mathbb{Q}}
(\mathrm{Tr}_{K_n/L_n}(\zeta_{\mathfrak{f}_{L_n}}))
=
\mathrm{Tr}_{K_n/\mathbb{Q}}(\zeta_{\mathfrak{f}_{L_n}})
=
\mu(\mathfrak{f}_{L_n}),
\]
where $\mu$ is the M\"{o}bius function.
We obtain the following theorem.
\begin{theo}\label{theo:MP-gauss}
Let $n$ be an integer with $3\nmid n$ or $n\equiv12\pmod{27}$,
and $a_0,a_1,m$ and $\e\in\{\pm1\}$ be integers defined in Theorem\,$1$.
Set $\mathfrak{f}_{L_n}=p_1\cdots p_t\ 
(p_1,\cdots,p_t\ \text{are\ distinct\ prime\ numbers})$,
$\rho_n^{(i)}:=\sigma^i (\rho_n) \ (i\in\mathbb{Z}/3\mathbb{Z})$
are the conjugates of $\rho_n$,
then we have the following under proper numbering.
\[
\eta_i
=
\frac{(-1)^t\e}{ec^2}(a_0\rho_n^{(i)}+a_1\rho_n^{(i+1)}+m)
=
\frac{(-1)^t\e}{ec^2}(a_1 {\rho_n^{(i)}}^2+(a_0-a_1n-a_1)\rho_n^{(i)}+m-2a_1),
\]
and the minimal polynomial of $\eta_i\ (i=0,1,2)$ is
\[
\begin{cases}
F_+(X), & \mathrm{if}\ \e=(-1)^t, 
\\
F_-(X), & \mathrm{if}\ \e=(-1)^{t+1},
\end{cases}
\]
where $F_\pm(X)$ is the polynomial given in Theorem\,$\ref{m.p.}$.
\end{theo}
For any integer $v$,
we have $\mathrm{Tr}_{L_n/\mathbb{Q}}(v+\rho_n)=3v+n$.
We obtain the following corollaries from Corollaries\,\ref{cor:3all NIB} and \ref{cor:f,g,h}.
We can see that Corollary\,\ref{cor:main} (1)  is equal to (\ref{eq:Lehmer-gauss}) in \S\,\ref{sect:intro},
 and Corollary\,\ref{cor:1} (1) and (2) are  equal to (\ref{eq:Lehmer-poly}) in \S\,\ref{sect:intro}.
\begin{cor}\label{cor:main}
Let $n\in\mathbb{Z},\ \mathfrak{f}_{L_n}=p_1\cdots p_t\ 
(p_1,\cdots,p_t\ \text{are\ distinct\ prime\ numbers})$, 
and $\rho_n^{(i)}$ $:=\sigma^i (\rho_n) \ (i\in\mathbb{Z}/3\mathbb{Z})$
be the conjugates of $\rho_n$.
\\
\begin{itemize}
\item[(1)]
If $3\nmid n$ and $\Delta_n$ is square-free, then we have
\[
\eta_i
=
(-1)^t\left(\frac{n}{3}\right)(v_n+\rho_n^{(i)}),
\]
where $v_n:=\left(\left(\dfrac{n}{3}\right)-n\right)/3$,
under proper numbering.
\\
\item[(2)]
If $n\equiv12 \pmod{27}$ and $\Delta_n/27$ is square-free, then we have
\[
\eta_i
=
\frac{(-1)^t}{9}(\rho_n^{(i)}-\rho_n^{(i+1)}+3)
=
\frac{(-1)^{t+1}}{9}({\rho_n^{(i)} }^2-(n+2)\rho_n^{(i)}-5),
\]
under proper numbering.
\end{itemize}
\end{cor}
\begin{cor}\label{cor:1}
Let $n\in\mathbb{Z},\ \mathfrak{f}_{L_n}=p_1\cdots p_t\ 
(p_1,\cdots,p_t\ \text{are\ distinct\ prime\ numbers})$,
and $f_\pm(X)$, $g_\pm(X)$, $h_\pm(X)$
are the polynomials given in Corollary\,$\ref{cor:f,g,h}$.
\\
\begin{itemize}
\item[(1)]
If $n\equiv1 \pmod{3}$ and $\Delta_n$ is square-free,
then the minimal polynomial of
$\mathrm{Tr}_{K_n/L_n}(\zeta_{\mathfrak{f}_{L_n}})$ is
\[
\begin{cases}
f_+(X), & \mathrm{if}\ t\ is\ even,
\\
f_-(X), & \mathrm{if}\ t\ is\ odd.
\end{cases}
\]
\\
\item[(2)]
If $n\equiv2 \pmod{3}$ and $\Delta_n$ is square-free,
then the minimal polynomial of
$\mathrm{Tr}_{K_n/L_n}(\zeta_{\mathfrak{f}_{L_n}})$ is
\[
\begin{cases}
g_+(X), & \mathrm{if}\ t\ is\ even,
\\
g_-(X), & \mathrm{if}\ t\ is\ odd,
\end{cases}
\]
\\
\item[(3)]
If $n\equiv12 \pmod{27}$ and $\Delta_n/27$ is square-free,
then the minimal polynomial of
$\mathrm{Tr}_{K_n/L_n}(\zeta_{\mathfrak{f}_{L_n}})$ is
\[
\begin{cases}
h_+(X), & \mathrm{if}\ t\ is\ even.
\\
h_-(X), & \mathrm{if}\ t\ is\ odd.
\end{cases}
\]
\end{itemize}
\end{cor}
\
\\
\begin{ex}\label{ex:table}
(1)\ 
The following table shows the Gaussian periods and their minimal
polynomials in the cubic fields $L_n$ for $n$ satisfying
$1\leq n\leq500,\ 3\nmid n$ and $\Delta_n\neq\mathfrak{f}_{L_n}$.
\begin{table}[H]
  \caption{ $1\leq n\leq500,\ 3\nmid n,\ \Delta_n\neq\mathfrak{f}_{L_n}$    }
 \begin{center}
 \begingroup
\renewcommand{\arraystretch}{1.5}
 \begin{tabular}{|c||c|c|c|c|}
 \hline
$n$ & $\Delta_n$  & $\mathfrak f_{L_n}$ & Gaussian period & minimal polynomial  \\  \hline  \hline 
$235$ & $13^2\cdot 331$ & $13\cdot 331$ & $\frac{1}{13}(\rho_n^2-239\rho_n +159)$ & $X^3-X^2-1434X+15937$ \\ \hline
$250$ & $7^2\cdot 1291$ & $7\cdot 1291$ & $-\frac{1}{7}(\rho_n^2-253\rho_n +79)$ & $X^3-X^2-3012X-32801$ \\ \hline
$269$ & $13^2\cdot 433$ & $13\cdot 433$ & $\frac{1}{13}(\rho_n^2-266\rho_n -446)$ & $X^3-X^2-1876X-22933$ \\ \hline
$271$ & $7\cdot 103^2$ & $7\cdot 103$ & $\frac{1}{103}(9\rho_n^2-2450\rho_n -616)$ & $X^3-X^2-240X-1175$ \\ \hline
$286$ & $7^3\cdot 241$ & $241$ & $-\frac{1}{49}(\rho_n^2-284\rho_n -367)$ & $X^3+X^2-80X+125$ \\ \hline
$299$ & $7^2\cdot 19 \cdot 97$ & $7\cdot 19 \cdot 97$ & $-\frac{1}{7}(\rho_n^2-302\rho_n +100)$ & $X^3+X^2-4300X-59249$ \\ \hline
$335$ & $7^2\cdot 2311$ & $7\cdot 2311$ & $-\frac{1}{7}(\rho_n^2-333\rho_n -451)$ & $X^3-X^2-5392X+85079$ \\ \hline
$356$ & $7\cdot 19 \cdot 31^2$ & $7\cdot 19\cdot 31$ & $-\frac{1}{31}(6\rho_n^2-2137\rho_n -1307)$ & $X^3+X^2-1374X-18019$ \\ \hline
$374$ & $37^2\cdot 103$ & $37\cdot 103$ & $-\frac{1}{37}(3\rho_n^2-1118\rho_n -1265)$ & $X^3-X^2-1270X+4799$ \\ \hline
$397$ & $7^3\cdot 463$ & $463$ & $\frac{1}{49}(\rho_n^2-400\rho_n +114)$ & $X^3+X^2-154X+343$ \\ \hline
$404$ & $7\cdot 13^2 \cdot 139$ & $7\cdot 13\cdot 139$ & $\frac{1}{13}(\rho_n^2-408\rho_n +263)$ & $X^3+X^2-4216X+76831$ \\ \hline
$433$ & $7^2\cdot 3853$ & $7\cdot 3853$ & $\frac{1}{7}(\rho_n^2-431\rho_n -577)$ & $X^3-X^2-8990X-175811$ \\ \hline
$446$ & $7^2\cdot 61 \cdot 67$ & $7\cdot 61 \cdot 67$ & $-\frac{1}{7}(\rho_n^2-449\rho_n +149)$ & $X^3+X^2-9536X-194965$ \\ \hline
$482$ & $7^2\cdot 13 \cdot 367$ & $7\cdot 13\cdot 367$ & $\frac{1}{7}(\rho_n^2-480\rho_n -647)$ & $X^3+X^2-11132X-249859$ \\ \hline
\end{tabular}
\endgroup
\end{center}
\end{table}
\bigskip
(2)\ 
The following table shows the Gaussian periods and their minimal
polynomials in the cubic fields $L_n$ for $n$ satisfying
$1\leq n\leq500$ and $n\equiv12\pmod{27}$ (in this case, always
$\Delta_n\neq\mathfrak{f}_{L_n}$).
\\
\begin{table}[H]
  \caption{ $1\leq n\leq500,\ n\equiv12\pmod{27}$   }
 \begin{center}
 \begingroup
\renewcommand{\arraystretch}{1.5}
 \begin{tabular}{|c||c|c|c|c|}
 \hline
 $n$ & $\Delta_n$  & $\mathfrak f_{L_n}$ & Gaussian period & minimal polynomial  \\  \hline  \hline 
 $12$ & $3^3\cdot 7$ & $7$ & $\frac{1}{9}(\rho_n^2-14\rho_n -5)$ & $X^3+X^2-2X-1$ \\ \hline
 $39$ & $3^3\cdot 61$ & $61$ & $\frac{1}{9}(\rho_n^2-41\rho_n -5)$ & $X^3+X^2-20X-9$ \\ \hline
 $66$ & $3^3\cdot 13^2$ & $13$ & $\frac{1}{117}(7\rho_n^2-464\rho_n -317)$ & $X^3+X^2-4X+1$ \\ \hline
 $93$ & $3^3\cdot 331$ & $331$ & $\frac{1}{9}(\rho_n^2-95\rho_n -5)$ & $X^3+X^2-110X-49$ \\ \hline
 $120$ & $3^3\cdot 547$ & $547$ & $\frac{1}{9}(\rho_n^2-122\rho_n -5)$ & $X^3+X^2-182X-81$ \\ \hline
 $147$ & $3^3\cdot 19 \cdot 43$ & $19\cdot 43$ & $-\frac{1}{9}(\rho_n^2-149\rho_n -5)$ & $X^3-X^2-272X+121$ \\ \hline
 $174$ & $3^3\cdot 7 \cdot 163$ & $7\cdot 163$ & $-\frac{1}{9}(\rho_n^2-176\rho_n -5)$ & $X^3-X^2-380X+169$ \\ \hline
 $201$ & $3^3\cdot 7^2 \cdot 31$ & $7\cdot 31$ & $\frac{1}{63}(5\rho_n^2-1006\rho_n -592)$ & $X^3-X^2-72X+225$ \\ \hline
 $228$ & $3^3\cdot 1951$ & $1951$ & $\frac{1}{9}(3\rho_n^2-230\rho_n -5)$ & $X^3+X^2-650X-289$ \\ \hline
 $255$ & $3^3\cdot 2437$ & $2437$ & $\frac{1}{9}(\rho_n^2-257\rho_n -5)$ & $X^3+X^2-812X-361$ \\ \hline
 $282$ & $3^3\cdot 13 \cdot 229$ & $13\cdot 229$ & $-\frac{1}{9}(\rho_n^2-284\rho_n -5)$ & $X^3-X^2-992X+441$ \\ \hline
 $309$ & $3^3\cdot 3571$ & $3571$ & $\frac{1}{9}(\rho_n^2-311\rho_n -5)$ & $X^3+X^2-1190X-529$ \\ \hline
 $336$ & $3^3\cdot 4219$ & $4219$ & $\frac{1}{9}(\rho_n^2-338\rho_n -5)$ & $X^3+X^2-1406X-625$ \\ \hline
 $363$ & $3^3\cdot 7 \cdot 19 \cdot 37$ & $7\cdot 19\cdot 37$ & $\frac{1}{9}(\rho_n^2-365\rho_n -5)$ & $X^3+X^2-1640X-729$ \\ \hline
 $390$ & $3^3\cdot 7 \cdot 811$ & $7 \cdot 811$ & $-\frac{1}{9}(\rho_n^2-392\rho_n -5)$ & $X^3-X^2-1892X+841$ \\ \hline
 $417$ & $3^3 \cdot 13 \cdot 499$ & $ 13\cdot 499$ & $-\frac{1}{9}(\rho_n^2-419\rho_n -5)$ & $X^3-X^2-2162X+961$ \\ \hline
 $444$ & $3^3 \cdot 7351$ & $7351$ & $\frac{1}{9}(\rho_n^2-446\rho_n -5)$ & $X^3+X^2-2450X-1089$ \\ \hline
 $471$ & $3^3 \cdot 8269$ & $8269$ & $\frac{1}{9}(\rho_n^2-473\rho_n -5)$ & $X^3+X^2-2756X-1225$ \\ \hline
 $498$ & $3^3\cdot 9241$ & $9241$ & $\frac{1}{9}(\rho_n^2-500\rho_n -5)$ & $X^3+X^2-3080X-1369$ \\ \hline
\end{tabular}
\endgroup
\end{center}
\end{table}
\end{ex}
\bigskip
\[
\textrm{Acknowledgements}
\]
\quad
The authors would like to thank Professor Hiroshi Tsunogai
 for his detailed explanation of this topic. 
 They would also like to thank Professors Masanari Kida and Shun'ichi Yokoyama
  for their advice on calculations with Magma.
\newpage

\end{document}